\documentclass[11pt]{article}
\usepackage{graphicx}
\usepackage{amsmath}
\usepackage{amsfonts}
\usepackage{epsfig}
\usepackage{setspace}
\usepackage{diagbox}
\usepackage{authblk}
\usepackage{caption}
\usepackage{rotating}
\usepackage{subcaption}
\usepackage{multirow}

\newcommand{\be}{\begin{equation}}
\newcommand{\ee}{\end{equation}}
\newcommand{\beqn}{\begin{eqnarray}}
\newcommand{\eeqn}{\end{eqnarray}}
\newcommand{\beqns}{\begin{eqnarray*}}
\newcommand{\eeqns}{\end{eqnarray*}}

\newcommand{\card}{\mbox{card}}
\newcommand{\Var}{\mbox{Var}}

\newcommand{\EE}{\ensuremath{{\mathbb E}}}

\newcommand{\fr}[1]{(\ref{#1})}

\newcommand{\Te}{\Theta}

\newtheorem{lemma}{Lemma}
\newtheorem{theorem}{Theorem}

\newtheorem{remark}{Remark}

\begin{document}

\title{\Large{\bf Minimax estimation for Varying Coefficient Model via Laguerre Series}}

\author[1]{Rida Benhaddou\thanks{benhaddo@ohio.edu}}
\author[2]{Khalid Chokri}
\author[1]{Jackson Pinschenat}
\affil[1]{Department of Mathematics, Ohio University, Athens, OH 45701}
\affil[2]{Modeling and Complex Systems Laboratory, Cadi Ayyad University, P.B. 549, Marrakech, Morocco}

\date{Draft 2}

\doublespacing
\maketitle
\begin{abstract}
We delve into the estimation of the functional coefficients and inference for varying coefficient model. Applying Laguerre series, we develop an estimator for the vector of functional coefficients that attains asymptotically optimal convergence rates in the minimax sense. These rates are derived for  functional coefficients that belong to Laguerre-Sobolev space. The method is based on approximating the functional coefficients using truncated Laguerre series and choosing empirical Laguerre coefficients that minimize the least squares criterion.  In addition, we establish the asymptotic normality of the estimator for the functional coefficients, construct their confidence intervals, and establish point-wise hypothesis tests about their true values. A simulations study is carried out to examine the finite-sample properties of the proposed methodology and to compare the performance to some of the existing methods. A real data set is considered as well, and results based on the proposed methodology are compared to those based on selected existing approaches.\\


{\bf AMS (2000) Subject Classification: 62G05, 62G20, 62G08 }
 \end{abstract} 

\section{Introduction}

Consider the  regression model 
\be
y_i=y(t_i, {\bf x}_{i}) =\sum^{\bf r}_{l=1}\beta_l(t_i)x_{li}+\sigma \varepsilon_i, \ \ \ i=1, 2, \cdots, n, \label{conveq}
\ee
where  for some fixed integer ${\bf r}\geq 2$, ${\bf x}_i=\{x_{1i}, x_{2i}, \cdots, x_{{\bf r}i}\}$ are $n$ copies of ${\bf X}= \left(X_1, X_2, \cdots, X_{\bf r}\right)$, $x_{i1}$, $x_{i2}$, $\cdots$, $x_{i{\bf r}}$, are $i.i.d.$ random variables with known compactly supported joint probability density $\mu$, $t_i$ are $i.i.d.$ random variables with probability density function $h$. The covariate $t$ plays the role of an effect-modifying factor that interacts with the effects of elements of ${\bf X}$ on the response $Y$. The quantities $\beta_l(.)$ are unknown measurable functions on ${\bf L}^2\left(0, \infty\right)$, they describe the interaction between the coefficients $\beta_l$ and the covariate $t$, and $\sigma>0$ is a known variance constant. $\{ \varepsilon_i\}_{i\geq1}$ is a stationary Gaussian sequence that is independent of $(t_i, {\bf x}_i)$ for any  $i\in \{1, 2, \cdots, n\}$.  Denote the $n$-dimensional vector with components $\varepsilon_i$ by  $\boldsymbol{\varepsilon}_n$, and let $\Sigma_n$ be its covariance matrix. The goal is to estimate the functional coefficients $\beta_l(t)$ based on the data points $(t_1, {\bf x}_1, y_1)$, $(t_2, {\bf x}_2, y_2)$, $\cdots$, $(t_n, {\bf x}_n, y_n)$.  A model of this sort is referred to as varying coefficients regression, or varying coefficients model (VCM). \\
VCM is a natural extension of the classical linear regression. It can be useful when one wishes to analyze how regression coefficients associated with the predictor variables change with respect to another variable, called the moderator or the effect-modifying covariate. This variable may be time, genetic factors, environmental factors, etc. This model is a very useful tool for exploring the dynamic patterns in many scientific disciplines, such as biomedical sciences, economics, finance, health sciences, to name a few. VCM has become very popular in data analysis since its introduction in Hastie and Tibshirani~(1993), and it has been studied extensively via a variety of nonparametric methods such as kernel-based methods (local linear, local polynomial estimation), smoothing splines and wavelets. This includes Hoover et al.~(1998), Wu, Chiang and Hoover~(1998), Fan and Zhang~(1999, 2008), Cai, Fan and Li~(2000), Chiang, Rice and Wu~(2001), Fan, Yao and Cai~(2003), Huang, Wu and Zhou~(2004), Zhou and You~(2004), and Pensky and Klopp~(2013, 2015). \\
In the present work, we propose to use Laguerre series to estimate the functional coefficients in a VCM. Laguerre series has been used to some extent in nonparametric estimation in selected statistical models, such as density estimation, and deconvolution. In particular, the application of Laguerre series  in nonparametric estimation can be found in Vareschi~(2013), Comte et al.~(2014), Comte and Genon-Catalot~(2015), Benhaddou, Pensky and Rajapaksha~(2017), Laverny et al.~(2021), Dussap~(2021), and Benhaddou and Connell~(2022). \\
The motivation of using Laguerre series in VCM stems from the fact that in some scientific applications, the effect-modifying covariates are defined on the positive real line, such as time, and since Laguerre series is specifically designed to represent and approximate square integrable functions on $[0, \infty)$, it may be able to capture dynamic, time-dependent relationships better than other methods. This may be suited to modeling longitudinal data dynamics or epidemiology, for instance, in studying how  the impact of pollution on health changes over time. The advantage of using Laguerre series is that the tuning parameters for the estimation procedure are basically integer-valued and therefore the minimization process of the MISE requires fewer viable candidates to choose from. Thus, it may have the computational advantage over kernel-based methods which require bandwidth selection, where bandwidths are non-negative real quantities, or smoothing spline methods which require the selection of smoothing parameters which are positive real-valued. \\
In this paper, we aim at exploring the application of Laguerre series representation to varying-coefficient models. We first establish the mathematical framework for such application. Under square-integrability condition, each functional coefficient can be represented as a Laguerre series. However, for estimation purposes, we approximate the functional coefficients by truncated Laguerre series and minimize the least-squares criterion with respect to the Laguerre coefficients. This minimization process yields a vector solution of empirical Laguerre coefficients, which  allows one to construct estimators for both individual functional coefficients $\beta_l(t)$, $l=1, 2, \cdots, {\bf r}$, and the vector of functional coefficients, $\boldsymbol{\beta}(t)=\left(\beta_1(t), \beta_2(t), \cdots, \beta_{\bf r}(t)\right)^T$.  The asymptotic properties of the estimators of both individual and the vector of functional coefficients are established. The truncation levels $M_1$, $M_2$, $\cdots$, $M_{\bf r}$, are chosen so that the mean squared error (MSE) for estimating the vector of functional coefficients is minimized. As a result, it is shown that the estimator for the vector of functional coefficients is asymptotical optimal in the minimax sense. In addition, the asymptotic normality of individual empirical functional coefficients is established, confidence intervals of the true functional coefficients are derived, and point-wise hypothesis tests are constructed. Furthermore, to study the finite-sample properties of the proposed procedure, a simulations study is conducted and our results are compared with selected existing methodologies. Finally, a real data set is analyzed using the proposed Laguerre series methodology and the local linear kernel estimator and their performances are compared. The results are also compared to those under baseline linear regression framework. 
 \section{Estimation Algorithm}
 Consider the orthonormal basis that consists of the system of Laguerre functions
\be\label{lague-def}
\phi_k(t)=e^{-t/2}L_k(t), \ \ k=0, 1, 2, \cdots, 
\ee
where $L_k(t)$ are Laguerre polynomials (see, e.g., Gradshtein and Ryzhik~(1980), Section 8.97). Since the functions $\phi_k(t)$, $k=0, 1, 2, \cdots$, form an orthonormal basis on $(0,\infty)$, the functions $\beta_l(t) \in {\bf L}^2\left(0, \infty\right)$ can be expanded over such basis as follows
\be
\beta_l(t)=\sum^{\infty}_{k=0}\theta_{lk}\phi_k(t), 
\ee
where 
\be\label{phi-it}
\theta_{lk}=\int^{\infty}_0\beta_l(t)\phi_k(t)dt. 
\ee
Since $t$ is random and follows the density $h$, consider instead the orthonormal basis 
\be\label{phi-tild}
\tilde{\phi}_k(t)=\frac{\phi_k(t)}{\sqrt{h(t)}}. 
\ee
\\
Now, consider the approximation of $\beta_l(t)$
\be
\tilde{\beta}_l(t)=\sum^{M_l-1}_{k=0}\theta_{lk}\tilde{\phi}_k(t). 
\ee
Keep in mind that
\be
\beta_l(t)=\sum^{\infty}_{k=0}\theta_{lk}\tilde{\phi}_k(t)=\tilde{\beta}_l(t)+\rho_l(t),
\ee
where 
\be
\rho_l(t)=\sum^{\infty}_{k=M_l}\theta_{lk}\tilde{\phi}_k(t).
\ee
Therefore, equation \fr{conveq} can be written in the form of a system of equations 
\be \label{Y-syst}
{\bf Y}=\Phi\Theta+\boldsymbol{\rho} +\sigma\boldsymbol{\varepsilon}_n,
\ee
where $\Phi$ is an $n\times \sum^{\bf r}_{l=1}M_l$ matrix such that the $i^{th}$ row is "$\phi_{10}(t_i)x_{1i}, \phi_{11}(t_i)x_{1i}, \cdots, \phi_{1(M_1-1)}x_{1i}$, $\phi_{20}(t_i)x_{2i}, \phi_{21}(t_i)x_{2i}, \cdots, \phi_{2(M_2-1)}, \cdots, \phi_{{\bf r}0}(t_i)x_{{\bf r}i}, \phi_{{\bf r}1}(t_i)x_{{\bf r}i}, \cdots, \phi_{{\bf r}(M_r-1)}x_{{\bf r}i}$, and 
$\Theta$ is the $\sum^{\bf r}_{l=1}M_l$-dimensional vector of coefficients $\theta_{10}, \theta_{11}, \theta_{12}, \cdots, \theta_{1(M_1-1)}, \theta_{20}, \theta_{21}, \cdots, \theta_{2(M_2-1)}, \cdots, \theta{{\bf r}(M_{\bf r}-1)}$, and $\boldsymbol{\rho}$ is the $n$-dimensional vector with $i^{th}$ element $\sum^{\bf r}_{l=1}\rho_l(t_i)x_{li}$. Notice, that if we left-multiply both sides of  \fr{Y-syst} by $\left[\Phi^T\Phi\right]^{-1}\Phi^T$, we obtain 
\be \label{Y-syst-phi}
\left[\Phi^T\Phi\right]^{-1}\Phi^T{\bf Y}=\Theta+\left[\Phi^T\Phi\right]^{-1}\Phi^T\boldsymbol{\rho} +\sigma\left[\Phi^T\Phi\right]^{-1}\Phi^T\boldsymbol{\varepsilon}_n,
\ee
Therefore, minimizing the least square criterion 
\beqns
L(\Theta)=\left[{\bf Y}-\Phi\Theta\right]^T\left[{\bf Y}-\Phi\Theta\right],
\eeqns
yields the solution
\be \label{Theta-hat}
\widehat{\Theta}=\left[\Phi^T\Phi\right]^{-1}\Phi^T{\bf Y}.
\ee
Solution \fr{Theta-hat} can be used to estimate individual functional coefficients, $\beta_l(t)$ $l=1, 2, \cdots, {\bf r}$,  or the vector of functional coefficients, $\boldsymbol{\beta}(t)=\left(\beta_1(t), \beta_2(t), \cdots, \beta_{\bf r}(t)\right)^T$. \\
\noindent{\bf Assumption A.1}. There exist constants $c_1$ and $c_2$,($0<c_1\leq c_2<\infty$), independent of $n$ such that
\be\label{spect-sigma}
c_1n^{1-\alpha}\leq \lambda_{\min}(\Sigma_n)\leq \lambda_{\max}(\Sigma_n)\leq c_2n^{1-\alpha}, \ \ 0<\alpha\leq 1,
\ee
where $\alpha$ is the long-memory parameter, with $\alpha=1$ corresponding to having short memory, and $\lambda_{\min}(\Sigma_n)$ and $\lambda_{\max}(\Sigma_n)$ are the smallest and the largest eigenvalues of $\Sigma_n$.  \\
\noindent{\bf Assumption A.2}. The random variables $X_l$, $l=1, 2, \cdots, {\bf r}$, are such that $\EE[X^2_{l}]<\infty$. \\
\noindent{\bf Assumption A.3}. The random vector ${\bf X}= \left(X_1, X_2, \cdots, X_{\bf r}\right)$ is such that the covariance matrix ${\bf W}=\EE\left[{\bf X}^T{\bf X}\right]$ has bounded eigenvalues, that is 
\be\label{eigen-w}
0<\omega_{0}=\lambda_{\min}\left({\bf W}\right)\leq \lambda_{\max}\left({\bf W}\right)=\omega_1<\infty.
\ee
\noindent{\bf Assumption A.4}. The density function $h$ is bounded away from zero, that is, for all $t\in(0, \infty)$, there exists $m_o>0$ such that $h(t)\geq m_o$. \\
\noindent{\bf Assumption A.5}. The functions $\beta_l(t)$ belong to Laguerre-Sobolev space, in particular, they are characterized by Laguerre coefficients $\theta_{lk}$ with the property
\be\label{lagsob}
{\bf B}^{\gamma_l}(A_l)=\left\{\beta_l\in L^2(0, \infty):\sum^{\infty}_{k=0}(k\vee 1)^{2\gamma_l}\theta_{lk}^2\leq A_l\right\}. 
\ee
  \begin{remark}	
 Functional spaces of type  \fr{lagsob} have been introduced in Bongioanni and Torrea~(2009) to study Laguerre operators, and the connection with Laguerre coefficients was established in Comte and Genon-Catalot~(2015). 
    \end{remark}
\section{Estimator of a Single Functional Coefficient, $\beta_l(t)$, and Asymptotic Analysis}
To estimate $\beta_l(t)$, we use 
\be \label{bet-l-est}
\widehat{\beta}_l(t)=\left(\tilde{\boldsymbol{\phi}}_l(t)\right)^T\widehat{\Theta},
\ee
where $\tilde{\boldsymbol{\phi}}_l(t)$ is the $\sum^{\bf r}_{l=1}M_l$-dimensional column vector whose elements are all zeros except in the $l^{th}$ block, in which the elements are the Laguerre functions $\tilde{\phi}_{lk}(t)$, with $k=0, 1, \cdots, M_{l}-1$, that is, 
\be
\left(\tilde{\boldsymbol{\phi}}_l(t)\right)^T=\left( 0, 0, \cdots, 0, \tilde{\phi}_{l0}(t), \tilde{\phi}_{l1}(t), \cdots, \tilde{\phi}_{l(M_l-1)}(t), 0, 0, \cdots, 0 \right).
\ee
To decide the truncation levels $M_l$, we need to study the mean square integrated error for \fr{bet-l-est}.
 \begin{lemma} \label{lem-ephiphi} Let {\bf Assumptions} {\bf A.2}-{\bf A.4} hold. Then, as $n\rightarrow \infty$
    \beqn
tr\left(\EE\left[\left[\Phi^T\Phi\right]^{-1}\right]\right)&\asymp& \frac{\sum^{\bf r}_{l=1}M_l}{n}, \label{tr-vec-phi}\\
tr\left(\EE\left[\left[\Phi^T\Phi\right]^{-1}\right]{\bf I}_l\right)&\asymp& \frac{M_l}{n}, \label{tr-l-phi}
\eeqn
where ${\bf I}_l$ is the  $\sum^{\bf r}_{d=1}M_d$-dimensional diagonal matrix whose $l^{th}$ block diagonal matrix of size $(M_l)\times(M_l)$ is the identity matrix and the rest of the diagonal elements are all zero. 
    \end{lemma}
 \begin{theorem}\label{th:upperbds-indiv} Let  $\widehat{\beta}_l(t)$ be defined in \fr{bet-l-est}, and let Assumptions {\bf A.1}, {\bf{A.2}}, {\bf{A.4}} and {\bf{A.5}} hold. Then, as $n\rightarrow \infty$,
\be \label{err-singl}
\EE\|\widehat{\beta}_l(t)-\beta_l(t)\|^2_2=O\left( AM_l^{-2\gamma_l}+\sigma^2n^{1-\alpha}\EE\left[\left(\tilde{\boldsymbol{\phi}}_l(t)\right)^T\left[\Phi^T\Phi\right]^{-1}\left(\tilde{\boldsymbol{\phi}}_l(t)\right)\right]_{h, \mu}\right).
\ee
  \end{theorem}
Notice that, since $\left(\tilde{\boldsymbol{\phi}}_l(t)\right)$ and $\left[\Phi^T\Phi\right]^{-1}$ are independent, we have by the orthogonality of Laguerre function basis 
\beqn
\EE\left[\left(\tilde{\boldsymbol{\phi}}_l(t)\right)^T\left[\Phi^T\Phi\right]^{-1}\left(\tilde{\boldsymbol{\phi}}_l(t)\right)\right]_{h, \mu}&=&\EE\left[tr\left(\left(\tilde{\boldsymbol{\phi}}_l(t)\right)^T\left[\Phi^T\Phi\right]^{-1}\left(\tilde{\boldsymbol{\phi}}_l(t)\right)\right)\right]_{h, \mu}\nonumber\\
&=&\EE\left[tr\left(\left[\Phi^T\Phi\right]^{-1}\EE\left[\left(\tilde{\boldsymbol{\phi}}_l(t)\right)\left(\tilde{\boldsymbol{\phi}}_l(t)\right)^T\right]_h\right)\right]_{h, \mu}\nonumber\\
&=&tr\left(\EE\left[\left[\Phi^T\Phi\right]^{-1}\right]_{h, \mu}{\bf I}_l\right).
\eeqn
Hence, by {\bf theorem \ref{th:upperbds-indiv}} and result \fr{tr-l-phi}, the optimal choice of $M_l$, $l=1, 2, \cdots, {\bf r}$, is
\be \label{M-optimal}
M^*_l=\left[A_l^2n^{\alpha}\right]^{\frac{1}{2\gamma_l+1}}. 
\ee
Notice that the truncation levels \fr{M-optimal} depend on $\gamma_l$, and therefore, estimator \fr{bet-l-est} is not adaptive. However, in practice, the selection of the truncation levels $M_1$, $M_2$, $\cdots$, $M_{\bf r}$, can be accomplished adaptively via cross validation. Observe that since $M_l$, $l=1, 2, \cdots, {\bf r}$, are integers, one can starts with lower integer values and keep increasing the values until the mean-integrated squared error (MISE) reaches the lowest point. Then use the combination $(M_1, M_2, \cdots, M_{\bf r})$ that yields the lowest MISE to choose a final estimator. 
\section{Estimator of the Vector of Functional Coefficients, ${\boldsymbol{\beta}}(t)$, and Asymptotic Analysis}
Let $\widehat{\boldsymbol{\beta}}(t)$ be the ${\bf r}$-dimensional vector such that 
\be\label{bet-vec-est}
\widehat{\boldsymbol{\beta}}(t)=\left(\widehat{\beta}_1(t), \widehat{\beta}_2(t), \cdots, \widehat{\beta}_{\bf r}(t)\right)^T,
\ee
 where $\widehat{\beta}_l(t)$, $l=1, 2, \cdots, {\bf r}$, are defined in \fr{bet-l-est}. Denote by $\Psi(t)$ the $(\sum^{\bf r}_{l=1}M_l)\times{\bf r}$  matrix whose columns are the vectors $\tilde{\boldsymbol{\phi}}_l(t)$, $l=1, 2, \cdots, {\bf r}$,  that is, 
\beqns
\Psi(t)=\left(\tilde{\boldsymbol{\phi}}_1(t), \tilde{\boldsymbol{\phi}}_2(t), \cdots, \tilde{\boldsymbol{\phi}}_{\bf r}(t)\right).
\eeqns
 \begin{theorem}\label{th:upperbds-vectr} Let $\widehat{\boldsymbol{\beta}}(t)$ be defined in \fr{bet-vec-est}, and let  Assumptions {\bf A.1}, {\bf{A.2}}, {\bf{A.4}} and {\bf{A.5}} hold. Then, as $n\rightarrow \infty$,
\be \label{err-full}
\EE\|\widehat{\boldsymbol{\beta}}(t)-\boldsymbol{\beta}(t)\|^2=O\left(\sum^{\bf r}_{l=1}A_l^2M_l^{-2\gamma_l}+\sigma^2n^{1-\alpha}tr\left(\EE\left[\left[\Psi(t)\right]^T\left[\Phi^T\Phi\right]^{-1}\left[\Psi(t)\right]\right]\right)\right). 
\ee
  \end{theorem}
Observe that by orthogonality of the elements of matrix $\Psi(t)$ and {\bf Lemma \ref{lem-ephiphi}}, $\EE\left[\left[\Psi(t)\right]^T\left[\Phi^T\Phi\right]^{-1}\left[\Psi(t)\right]\right]$ will have $\sum^{\bf r}_{l=1}M_l$ nonzero elements, and therefore by choosing $(M_1, M_2, \cdots, M_{\bf r})$ according to \fr{M-optimal},  \fr{err-full} gives
\be\label{uppr-bnd}
\EE\|\widehat{\boldsymbol{\beta}}(t)-\boldsymbol{\beta}(t)\|^2\leq C\sum^{\bf r}_{l=1}A^2_l \left[\frac{\sigma^2}{A_l^2n^{\alpha}}\right]^{\frac{2\gamma_l}{2\gamma_l+1}}. 
\ee
\section{Minimax Lower Bounds for Estimating Vector of Functional Coefficients, $\widehat{\boldsymbol{\beta}}(t)$}
We define the minimax $L^2$-risk over a set $\Theta$ as 
\beqns
\mathcal{R}(\Theta)=\inf_{\tilde{\bf f}}\sup_{{\bf f}\in \Theta}\EE\|\tilde{\bf f}-{\bf f}\|^2,
\eeqns
where the infimum is taken over all possible vector of estimators $\tilde{\bf f}$ of ${\bf f}$. 
  \begin{theorem}\label{th:lowerbds} Let  {\bf Assumptions} {\bf A.1}, {\bf A.2}, {\bf A.3} and {\bf A.5} hold. Let $\boldsymbol{\gamma}=(\gamma_1, \gamma_2, \cdots, \gamma_{\bf r})$. Then, as $n\rightarrow \infty$,
 \be \label{lowerbds1}
 \mathcal{R}({{B^{\boldsymbol{\gamma}}}}(A))\geq C\sum^{\bf r}_{l=1}A^2_l \left[\frac{\sigma^2}{A_l^2n^{\alpha}}\right]^{\frac{2\gamma_l}{2\gamma_l+1}}. 
    \ee
 \end{theorem}
  \begin{remark}	
 Based on \fr{uppr-bnd} and  \fr{lowerbds1}, estimator  \fr{bet-vec-est} for vector ${\boldsymbol{\beta}}(t)$ with choices \fr{M-optimal} is asymptotically optimal in the minimax sense. However, the estimator is not adaptive since the choices $M_l$ depend on the smoothness of the functional coefficients, $\beta_l(t)$. 
    \end{remark}
       \begin{remark}	
Notice that the estimator \fr{bet-l-est} for individual functional coefficients, $\beta_l(t)$, may not be optimal in the sense of Fan and Zhang~(1999), unless the regularities $\gamma_l$, $l=1, ,2 \cdots, {\bf r}$, are the same. However, in the present work, we define optimality with respect to the vector of coefficients, ${\boldsymbol{\beta}}(t)$. 
    \end{remark}
   Fan and Zhang~(1999) introduced local linear estimator for the varying coefficient models where the functional coefficients are approximated locally via linear functions of the effect-modifying covariate and the least squares criterion is minimized. They showed that the estimators for the individual functional coefficients are not asymptotically optimal in minimax sense unless the functional coefficients have the same regularity.  They also concluded that optimality may be achieved using a certain multiple-step procedure. \\
Chiang et al.~(2001) used smoothing splines and penalized least squares to estimate the functional coefficients in varying coefficient model and cross validation to choose the smoothing parameters, and established the consistency of their proposed estimator. \\
Pensky and Klopp~(2013) investigated varying coefficient model in a high-dimensional setup and developed minimax theory in a non-asymptotic fashion. Their method relies on a general series solution and a penalization called group-LASSO, in which the covariates are assumed to cluster in groups. Therefore, their work does not directly relate to ours since the number of covariates in our setting is fixed and finite, that is, ${\bf r}<\infty$. \\
The advantage of the proposed Laguerre series-based method is that the tuning parameters, $M_1$, $M_2$, $\cdots$, $M_{\bf r}$, are integer-valued and therefore the minimization process of the MISE requires fewer viable candidates to choose from, unlike the kernel-based methods which rely on selecting bandwidths, $h_1$, $h_2$, $\cdots$, $h_{\bf r}$, which have strictly positive real values, or smoothing spline methods which rely on selecting the smoothing parameters, $\lambda_1$, $\lambda_2$, $\cdots$, $\lambda_{\bf r}$, which are positive real-valued. Therefore, the proposed Laguerre approach  may have a computational advantage over the kernel-based and smoothing spline methods in some specific circumstances. 
  \section{ Asymptotic Normality of the Estimator for  $\beta_l(t)$}
  \begin{theorem}\label{thnor}  (Asymptotic Normality).
Assume that the functional coefficients $\beta_l(t)$ belong to a Laguerre-Sobolev space with parameter $\gamma > 0$, the design matrix $\Phi$ satisfies the invertibility condition with high probability, and the errors $\varepsilon_i$ are stationary Gaussian with covariance matrix $\Sigma$ and long-memory parameter $\alpha \in (0,1)$. Let $\hat{\beta}_l(t)$ be the adaptive thresholding estimator defined by
\beqns
\hat{\beta}_l(t) = \left( \tilde{\boldsymbol{\phi}}_l(t) \right)^T \hat{\Theta},
\eeqns
where $\hat{\Theta} = (\Phi^T \Phi)^{-1} \Phi^T Y$, and the truncation level $M_l$ is chosen as $M_l \asymp n^{1 / (2\gamma+1)}$ to achieve minimax optimality.
Then, as $n \to \infty$,
\beqns
\sqrt{n^{\alpha}} \left( \widehat{\beta}_l(t) - \beta_l(t) \right) \xrightarrow{d} \mathcal{N} \left( 0, \sigma_l^2(t) \right),
\eeqns
where the asymptotic variance is
\beqns
\sigma_l^2(t) = \pi_{\alpha} \tilde{\boldsymbol{\phi}}_l(t) ^T \Gamma^{-1}   \tilde{\boldsymbol{\phi}}_l(t),
\eeqns
$\Gamma = \mathbb{E} [X_{l1}^2 \tilde{\phi}_k(t_1) \tilde{\phi}_j(t_1)]$, and $\xrightarrow{d}$ denotes convergence in distribution.
 \end{theorem}
Note that $\Gamma$ is a diagonal matrix, since $\tilde{\phi}_k(t)$ are orthonormal. Therefore, 
\be \label{asymp-var}
\sigma_l^2(t) = \pi_{\alpha} \tilde{\phi}_l(t) ^T \Gamma^{-1}   \tilde{\phi}_l(t)=\frac{\pi_{\alpha}}{\EE[X_{l1}^2]} \sum^{M_l-1}_{k=0}\tilde{\phi}^2_k(t).
\ee
 \section{Confidence intervals for  $\beta_l(t)$}
 Our results allow us to construct confidence intervals for the true function  $\beta_l(t)$.
Towards this end, we infer from {\bf Theorem \ref{thnor}} that, we have the confidence interval, as $n \to \infty$, 
\beqns
\hat{\beta}_l(t)\pm \phi_{1-\frac{\alpha}{2}}n^{-\alpha/2}\sigma_l(t),
\eeqns
where $\phi_{1-\frac{\alpha}{2}}$ is the lower quantile of order $(1-\frac{\alpha}{2})$ of the standard normal law.
\remark
If $\Sigma_n = \sigma^2 I$ (i.i.d. errors), the variance simplifies to
\beqns
\sigma_l^2(t) = \sigma^2 \left( \tilde{\boldsymbol{\phi}}_l(t) \right)^T \left( \mathbb{E} \left[ \frac{\Phi^T \Phi}{n} \right]^{-1} \right) \left( \tilde{\boldsymbol{\phi}}_l(t) \right),
\eeqns
which aligns with standard regression theory. 
\remark  A higher regularity ($\gamma_l$ large) accelerates the bias decay but imposes stricter smoothness constraints on $\beta_l(t)$.
\section{Hypothesis Test about the true $\beta_l(t)$}
Following {\bf Theorem \ref{thnor}}, we have 
\beqns
\sqrt{n^{\alpha}} \left( \hat{\beta}_l(t) - \beta_l(t) \right) \xrightarrow{d} \mathcal{N} \left( 0, \sigma_l^2(t) \right),
\eeqns
where $ \sigma_l^2(t) = \pi_{\alpha} \, \tilde{\phi}_l(t)^T \Gamma^{-1} \tilde{\phi}_l(t)$, and  $\Gamma=\mathbb{E} \left[ X_{l1}^2 \tilde{\phi}_k(t_1) \tilde{\phi}_j(t_1) \right]$. Let $H_o$ be the null hypothesis defined by
\beqns
 H_o : \beta_l(t_o) = \beta_{l,o}(t_o) \quad \text{for } a\ {fixed} \quad t_o,
\eeqns
where \(\beta_{l,o}(t_o)\) is a specific value. Then, under $H_o$
\beqns
T_n(t_o) = \frac{\sqrt{n^{\alpha}} \left( \hat{\beta}_l(t_o) - \beta_{l,o}(t_o) \right)}{\hat{\sigma}_l(t_o)} \xrightarrow{d} \mathcal{N}(0, 1),
\eeqns
where  \(\hat{\sigma}_l^2(t_o)\) is a consistent estimator of \(\sigma_l^2(t_o)\). We can estimate $\Gamma$ by
\beqns
\hat{\Gamma}_n = \frac{1}{n} \sum_{i=1}^n X_{l,i}^2 \tilde{\phi}_l(t_i) \tilde{\phi}_l(t_i)^T,
\eeqns
and then, estimate $\sigma_l^2(t)$ by
\beqns
\hat{\sigma}_l^2(t) = \hat{\pi}_{\alpha} \, \tilde{\phi}_l(t)^T \hat{\Gamma}_n^{-1} \tilde{\phi}_l(t),
\eeqns
where \(\hat{\pi}_{\alpha}\) is any consistant estimator of the finite long memory constant $\pi_{\alpha}$,(e.g. the Whittle estimator). For testing the point-wise null hypothesis
\beqns
H_o : \beta_l(t_o) = \beta_{l,o}(t_o) \quad \text{versus} \quad H_1 : \beta_l(t_o) \neq \beta_{l,o}(t_o),
\eeqns
at a fixed point $t_o$, the rejection region at significance level $\alpha \in (0,1)$ is defined as:
\beqns
\mathcal{R}_\alpha = \left\{ (y_i, X_{l,i}, t_i)_{i=1}^n \in \mathbb{R}^{3n} : \left| T_n(t_o) \right| > z_{\alpha/2} \right\}.
\eeqns
So, we reject $H_o$ at significance level $\alpha$ whenever
\beqns
|T_n(t_o)| > z_{\alpha/2} = \Phi^{-1}(1 - \alpha/2).
\eeqns
Consider the local alternatives of the form 
\beqns
H_1^{(n)} : \beta_l(t) = \beta_{l,o}(t) + \frac{\delta}{\sqrt{n^{\alpha}}}, \quad \text{where} \quad  \delta\ne 0.
\eeqns
Therefore,  under $H_1^{(n)}$, one has
\beqns
T_n(t_o) \xrightarrow{d} \mathcal{N}\left( \frac{\delta}{\sigma_l(t_o)}, 1 \right).
\eeqns
Consequently, the asymptotic power function is
\beqns
\beta_n(\delta) = P\left( |T_n(t_o)|> z_{\alpha/2}\Big|H_1^{(n)}\right) = 1 - \Phi\left( z_{\alpha/2} - \frac{\delta}{\sigma_l(t_o)} \right) + \Phi\left( -z_{\alpha/2} - \frac{\delta}{\sigma_l(t_o)} \right).
\eeqns
\section{Simulations Study}
 To assess the performance of estimators \fr{bet-l-est} and \fr{bet-vec-est}  in a finite sample setting and to compare it to existing methods, we carry out a limited simulation study with $r=2$. We evaluate the mean integrated square error (MISE) $\EE \| \hat{y}-y\|^2$ of the  estimator. The simulation is completed with R.

\begin{enumerate}

\item We first generate  $x_{1i}$ and $x_{2i}$ using two distinct normal density functions $N(200, 20)$ and $N(45, 5)$, respectively. The $t_i$ are generated using an exponential distribution with mean equal to 0.25.

\item We generate the data using equation \fr{conveq} with test functions $\beta_1(t_{i}) = t^{5/2}_i e^{-(t_i-3)}$, $\beta_2(t_{i})=\frac{t_i}{(t^2_{i} +1 )^4}$,  and $\beta_3(t_i)=\frac{1}{e^{t_i}+e^{2t_i}}$, $i = 1,2, . . . ,n$, with $n=400, 800, 1200$.

\item We simulate the error $\{\varepsilon_i\}$ in \fr{conveq} using standard normal, and we use $\sigma=0.0001$. 

\item To choose $M_1$ and $M_2$, we use leave-one-out cross validation.

\item We compute  the averages of the errors over $n_0=1000$ repeated simulations by
\be \label{misehat}
\widehat{\text{MISE}}=\frac{1}{n_0}\sum\limits_{i=1}^{n_0} \EE  \left\| \hat{y_i}-y_i \right\|^2.
\ee
\item The proposed Laguerre methodology (GL-VCM) is compared to the Local Linear Kernel method (LL-VCM) (Fan and Zhang~(1999)) and the Nadaraya-Watson method (NW-VCM) (Park et al.~(2015)). To choose the bandwidth of each, cross-validation will be employed. 
\end{enumerate}
\begin{table}[h]
\centering
\caption{Comparison of Basis Complexity ($M, h$) and MISE (0.0001 Noise Scaling)}
\label{tab:final_simulation_results}
\begin{tabular}{lcccccccccc}
\hline
\textbf{Model} & $n$ & \multicolumn{3}{c}{$\beta_1, \beta_2$} & \multicolumn{3}{c}{$\beta_2, \beta_3$} & \multicolumn{3}{c}{$\beta_1, \beta_3$} \\

& & $M_1/h$ & $M_2/h$ & MISE & $M_1/h$ & $M_2/h$ & MISE & $M_1/h$ & $M_2/h$ & MISE \\
\hline
{GL-VCM} 
& 400  & 4.275 & 4.531 & 2.0288 & 4.986 & 3.119 & 38.6853 & 4.914 & 4.684 & 0.0041 \\
 & 800  & 5.511 & 4.866 & 1.6164 & 5.648 & 3.295 & 32.5238 & 5.639 & 4.578 & 0.0029 \\
& 1200 & 5.589 & 5.086 & 1.4573 & 5.749 & 3.770 & 29.3013 & 5.892 & 4.680 & 0.0019 \\
\hline
{LL-VCM} 
& 400  & 0.223 & 0.223 & 6.3271 & 0.198 & 0.198 & 124.7308 & 0.188 & 0.188 & 0.6808 \\
& 800  & 0.203 & 0.203 & 5.0563 & 0.184 & 0.184 & 93.7865  & 0.217 & 0.217 & 0.9448 \\
& 1200 & 0.196 & 0.196 & 4.4109 & 0.176 & 0.176 & 80.6365  & 0.228 & 0.228 & 1.0500 \\
\hline
{NW-VCM} 
& 400  & 0.063 & 0.063 & 25.7784 & 0.063 & 0.063 & 432.8071 & 0.088 & 0.088 & 46.8313 \\
& 800  & 0.056 & 0.056 & 25.1055 & 0.056 & 0.056 & 419.3375 & 0.080 & 0.080 & 46.9144 \\
& 1200 & 0.053 & 0.053 & 24.8080 & 0.053 & 0.053 & 415.7049 & 0.075 & 0.075 & 46.9144 \\
\hline
\end{tabular}
\end{table}
Table 1 outlines the behavior of the MISE depending on the chosen test functions and M1, M2, and h. The M1, M2, and h that are listed represent averages of those choices over the 1000 simulation runs. As you can see for the GL-VCM, the MISE decreases as the sample size $n$ increases. In addition, the MISE is substantially lower for the proposed GL-VCM approach compared to the two local kernel methods, demonstrating the Laguerre method's strength in estimating the functional effects in VCM when these effects are square-integrable on $[0, \infty)$. Furthermore, the choices M1 and M2 increase as $n$ increases, allowing the estimates to have higher resolution. This further increases the appeal of the method. \\
Figure 1 shows the plots of pairs of the actual functional coefficients $\beta_l(t)$, the shaded curves, versus the estimated ones (the bold curves) based on the three methods. ${\bf f}1$, ${\bf f}2$ and ${\bf f}3$ represent, respectively, the test functions $\beta_1(t)$, $\beta_2(t)$, and $\beta_3(t)$.  As you can see, the proposed GL-VCM approach shows the most accuracy compared to the other two kernel-based methods. The $\beta_1(t)$ was chosen to have a lower regularity, but still GL-VCM method does a very good job in detecting the detailed features of the functional coefficients. The $\beta_2(t)$ has the slowest rate of decay compared to $\beta_1(t)$ and $\beta_3(t)$, yet GL-VCM managed to estimate it with great accuracy while the other two methods failed.  
 \begin{figure}[htbp]
\caption{Actual versus estimated $\beta_l$ based on the three different methods}
\rotatebox{-90}
{\includegraphics[width=5in]{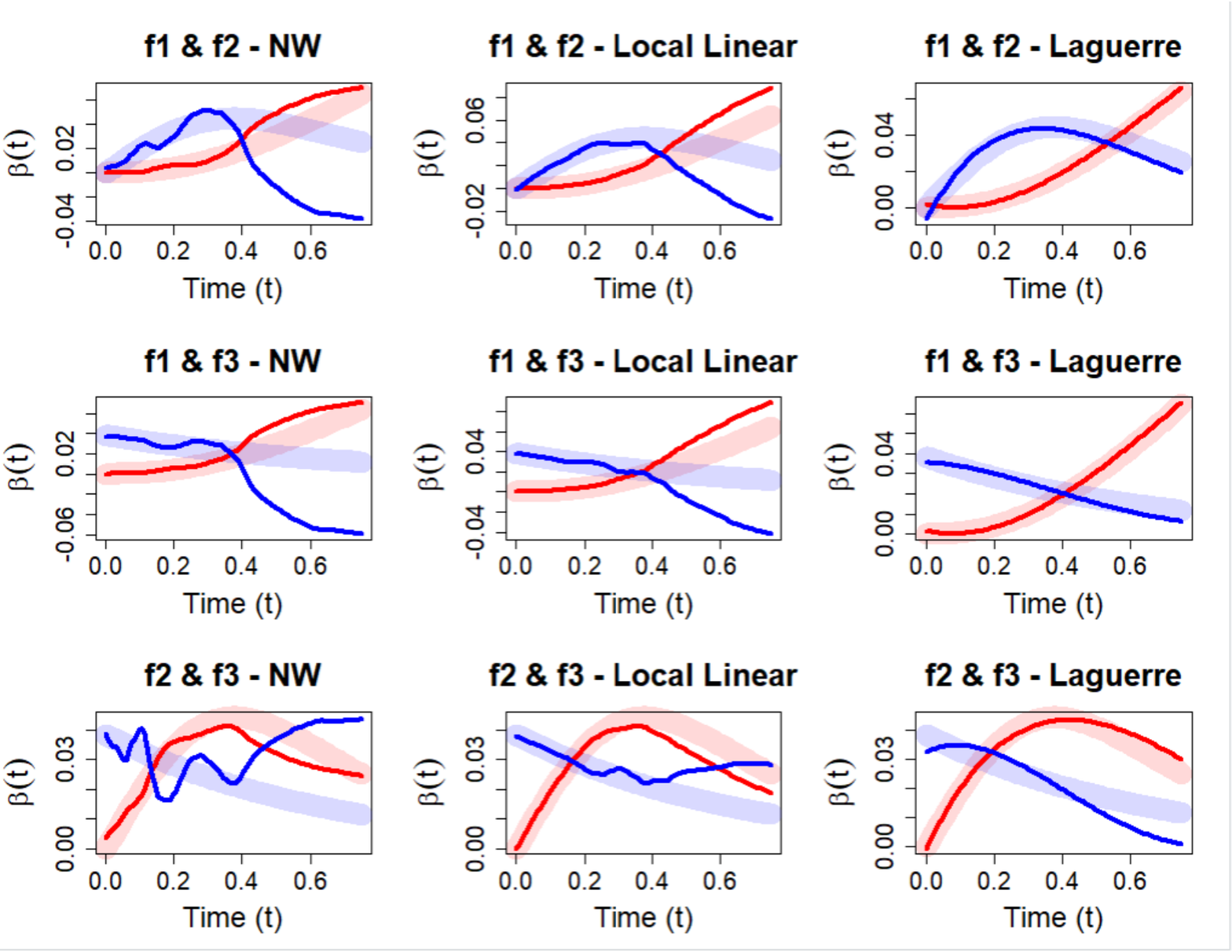}}
\end{figure}
 \section{Analysis of Real Data}
  To further the analysis of the model, a dataset called "SAheart" from the R package "ElemStatLearn" will be utilized. It makes use of sample data from 462 males in West Cape, South Africa, specifically a high-risk region for heart disease. The data originally had the purpose of analyzing 9 different variables effect on the presence or absence of coronary heart disease(Chd) in a subject. The data frame includes factors such as systolic blood pressure(sbp), low density lipoprotein cholesterol(ldl), adiposity, Type-A behavior score (typea), age, obesity, and alcohol. \\
 In this section, we will compare several estimation methods under the varying coefficient regression setup using the heart data as well as a linear regression without varying coefficients. The data will fit the varying coefficient model,
 \beqns
        obesity = \beta_0(age/100) + chd * \beta_1(age/100) + typea * \beta_2(age/100) + adiposity * \beta_3(age/100) + \varepsilon,
 \eeqns
        for the Generalized Laguerre, GL-VCM, Local-Linear, LL-VCM, and Smoothing Spline VCM (SS-VCM), while the data will fit the model
\beqns
        obesity = \beta_0 + chd * \beta_1 + typea * \beta_2 + adiposity * \beta_3 + \varepsilon,
\eeqns
        for the linear regression. A training and testing split with cross validation was used to find the optimal bandwidth $h$ to use for minimizing the error in the local linear kernel estimator for varying coefficient model. A grid search method was used to find the optimal truncation level $M$ in the generalized Laguerre method. The MSE was lowest at h = .133, for the local linear kernel method while it was the lowest at M1 = 3, M2 = 3, M3 = 5, and M4 = 3, for the generalized Laguerre method. Table 2 outlines the model performance.
\begin{table}[h]
\centering
\caption{Model Performance Comparison: Obesity Prediction}
\label{tab:model_comparison}
\begin{tabular}{lcccc} 
\hline
\textbf{Model} & \textbf{(M1-M4)/h} & \textbf{$R^2$} & \textbf{MSE} & \textbf{AIC} \\ \hline
Linear Regression    & N/A     & 0.5341 & 8.2549 & 983.19 \\ \hline
Generalized Laguerre & 3/3/5/3 & 0.5926 & 7.2171 & 941.12 \\ \hline
Local Linear         & 0.133   & 0.6005 & 7.0772 & 940.39 \\ \hline
Smoothing Spline     & N/A     & 0.5886  & 7.0788 & 932.73 \\ \hline
\end{tabular}
\end{table}
\\
\begin{figure}
    \centering
     \caption{Laguerre VCM Coefficient Estimates and confidence bands versus age}
    \includegraphics[width=0.5\linewidth]{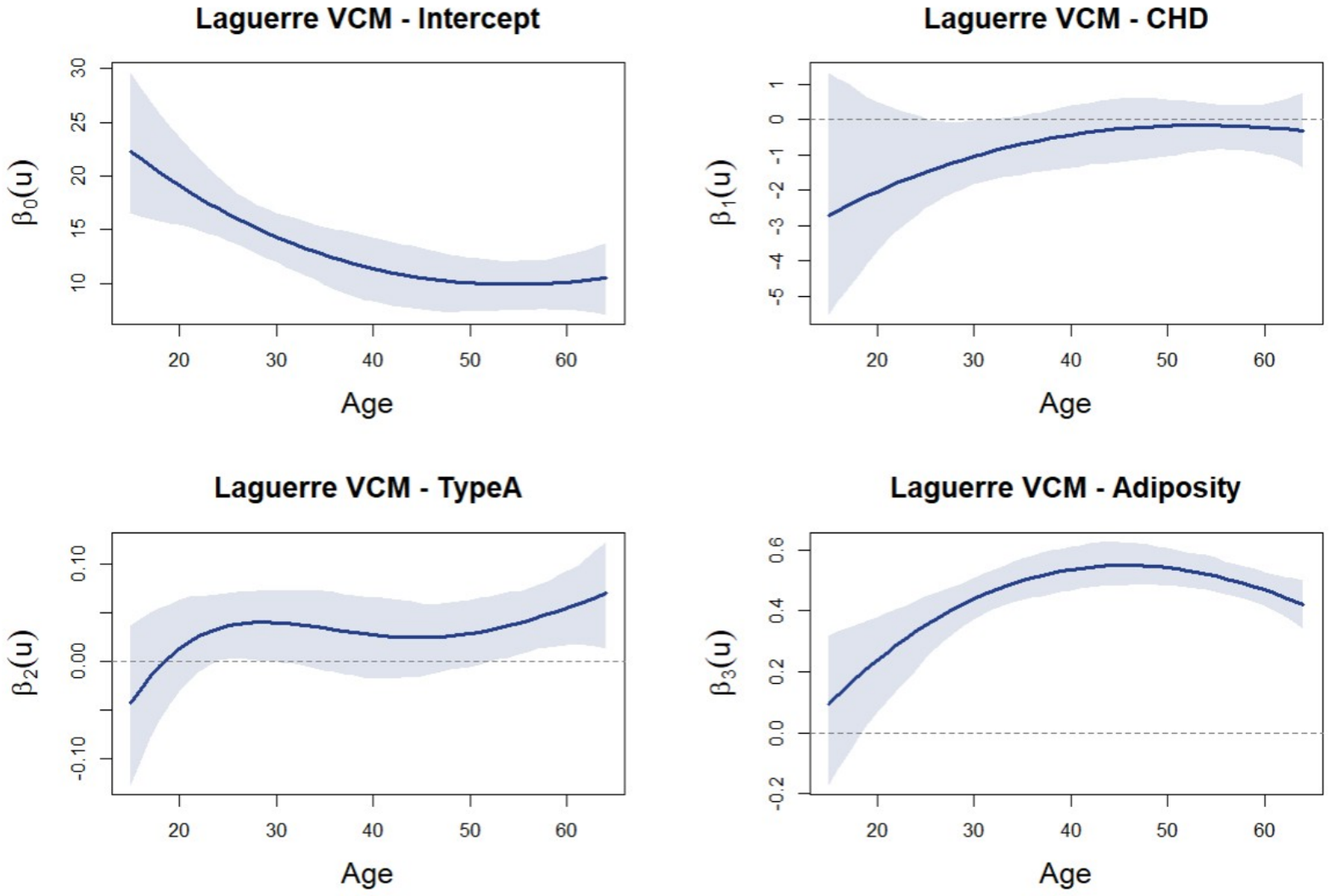}
    \label{fig:placeholder}
\end{figure}
\begin{figure}
    \centering
     \caption{Smoothing Spline VCM Coefficient Estimates Over Time}
    \includegraphics[width=0.5\linewidth]{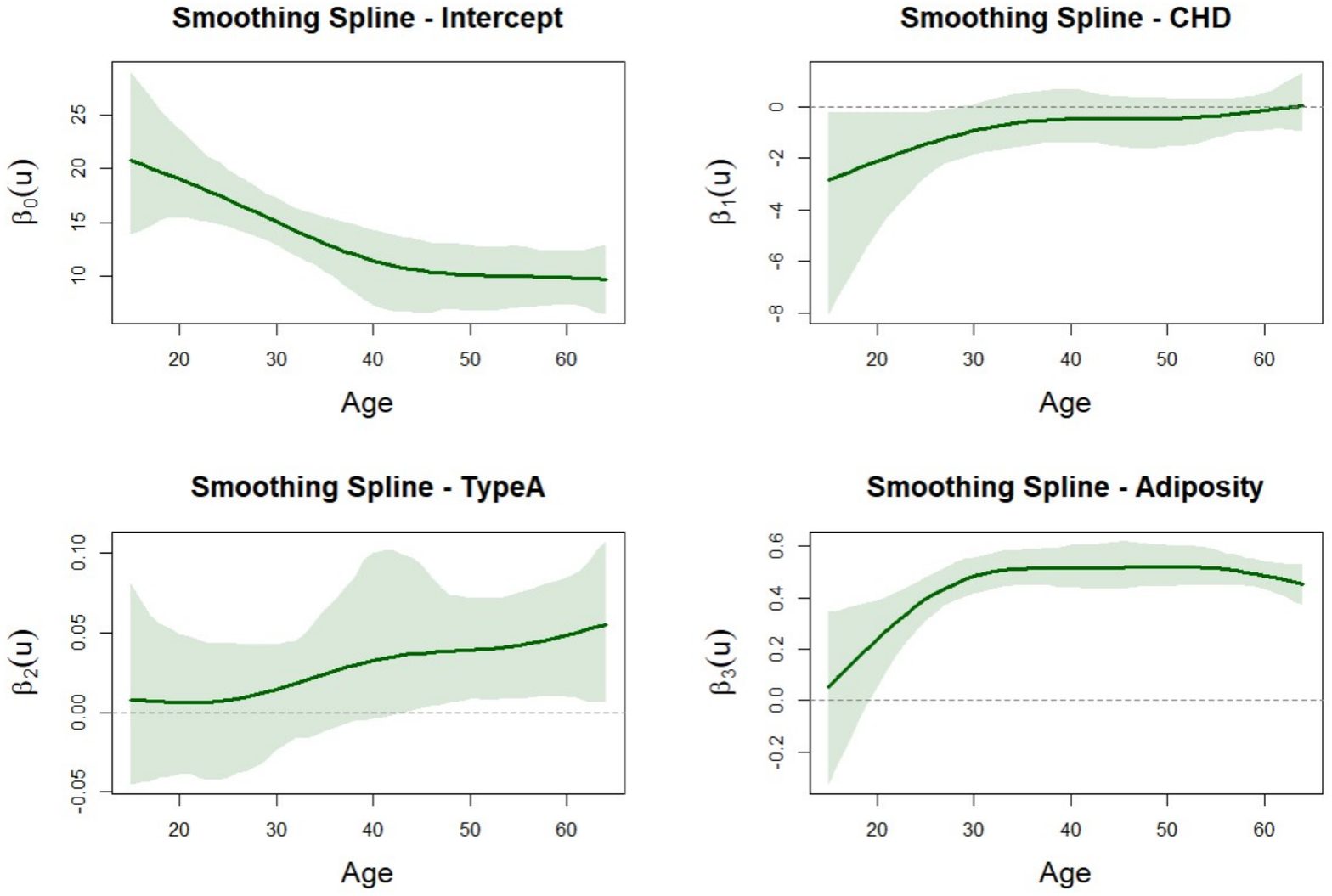}
 \label{SSbetacoefficients}
\end{figure}
Therefore, on the SAheart dataset, the Laguerre Method performs similarly to the local linear kernel method and the smoothing spline, but better than the baseline model, linear regression. \\
Using the bootstrap method, 500 different samples can be created to build a 95 percent confidence interval for the beta coefficients. Figures \ref{fig:placeholder}, \ref{SSbetacoefficients} and \ref{KernelBetaCoefficients} show the estimated functional beta coefficients based on the three methods as age changes with the shaded regions representing the confidence bands of the true functional coefficients. All three methods clearly capture the dynamic change in effects as age increases, nevertheless, Laguarre method shows relatively smoother curves and finer confidence bands. Notice that we can even observe which age range/s the factors has statistically significant effect on blood pressure. For instance, the effect of ChD is only statistically significant for ages ranging between around 25 and 32, and not significant for other age ranges.  Type A effect is statistically significant for ages between 22 and 30, and 52 and above. Whereas for Adiposity, the effect is statistically significant for all ages represented in the data set. 
\begin{figure}
    \centering
      \caption{Local Linear Kernel VCM Coefficient Estimates Over Time}
    \includegraphics[width=0.5\linewidth]{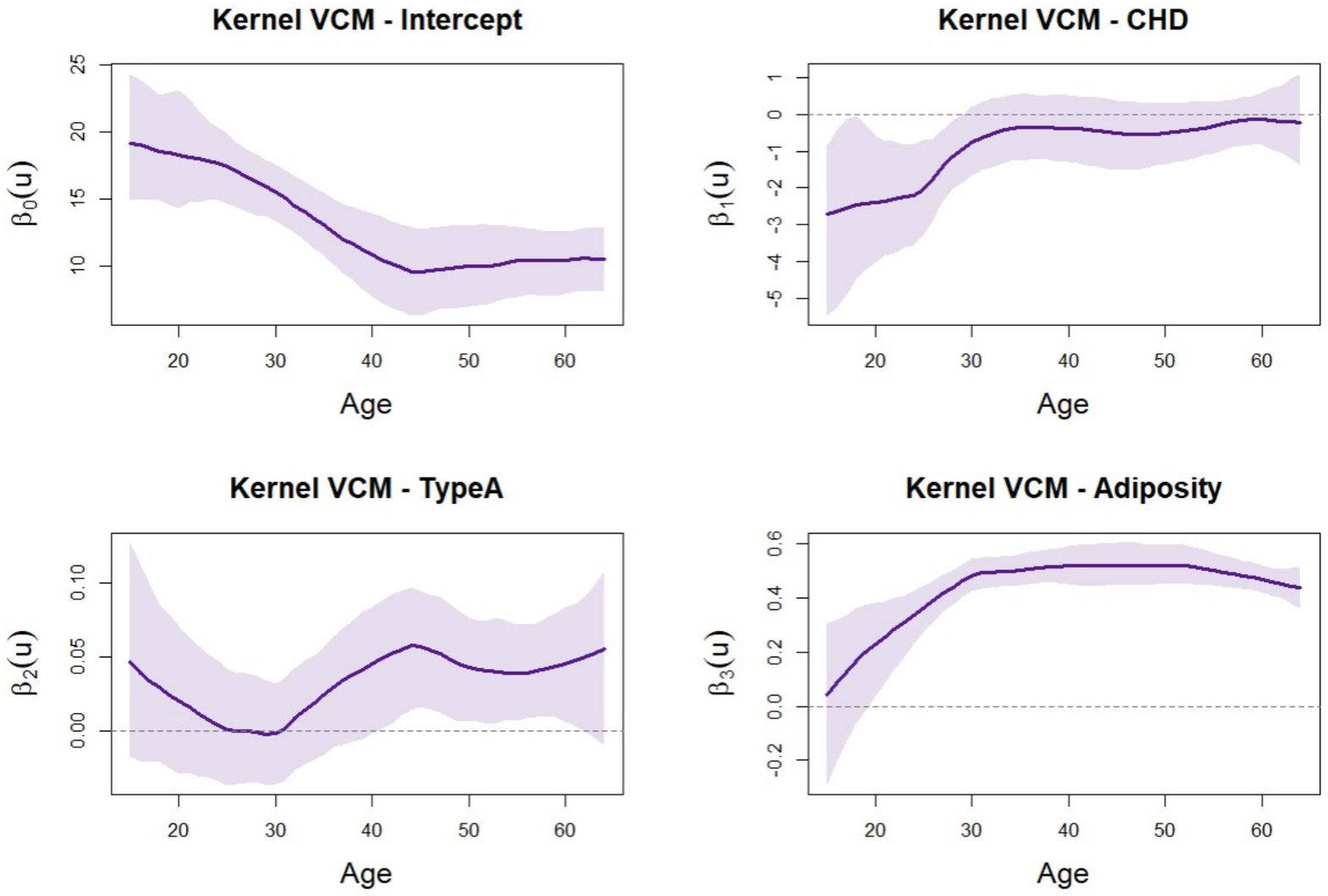}
   \label{KernelBetaCoefficients}
\end{figure}
\subsection{Residual Analysis}
\begin{figure}
    \centering
    \includegraphics[width=1\linewidth]{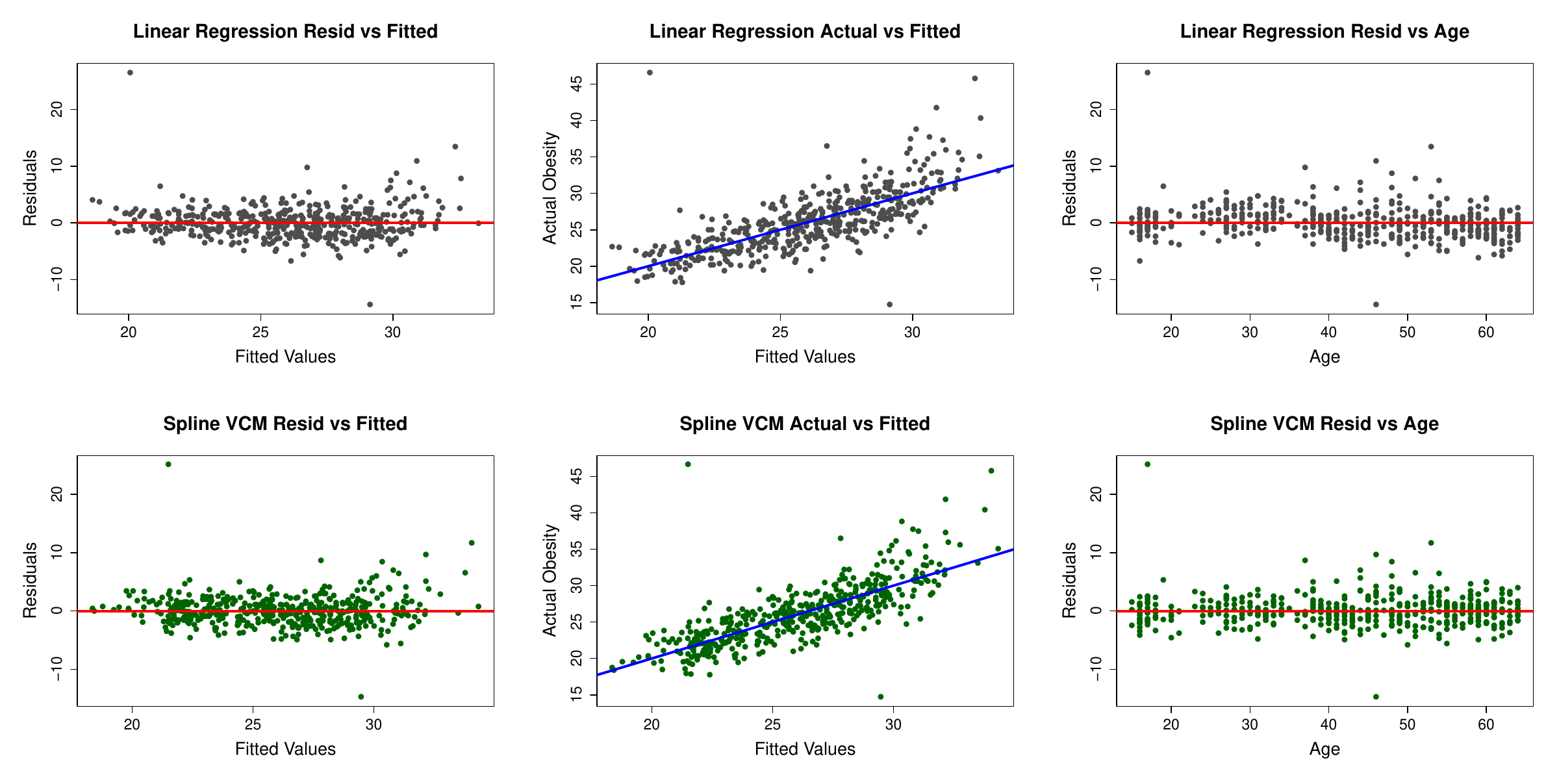}
    \caption{Linear Regression and Smoothing Spline VCM Residual Analysis}
    \label{SSandLMResiduals}
\end{figure}
\begin{figure}
    \centering
    \includegraphics[width=1\linewidth]{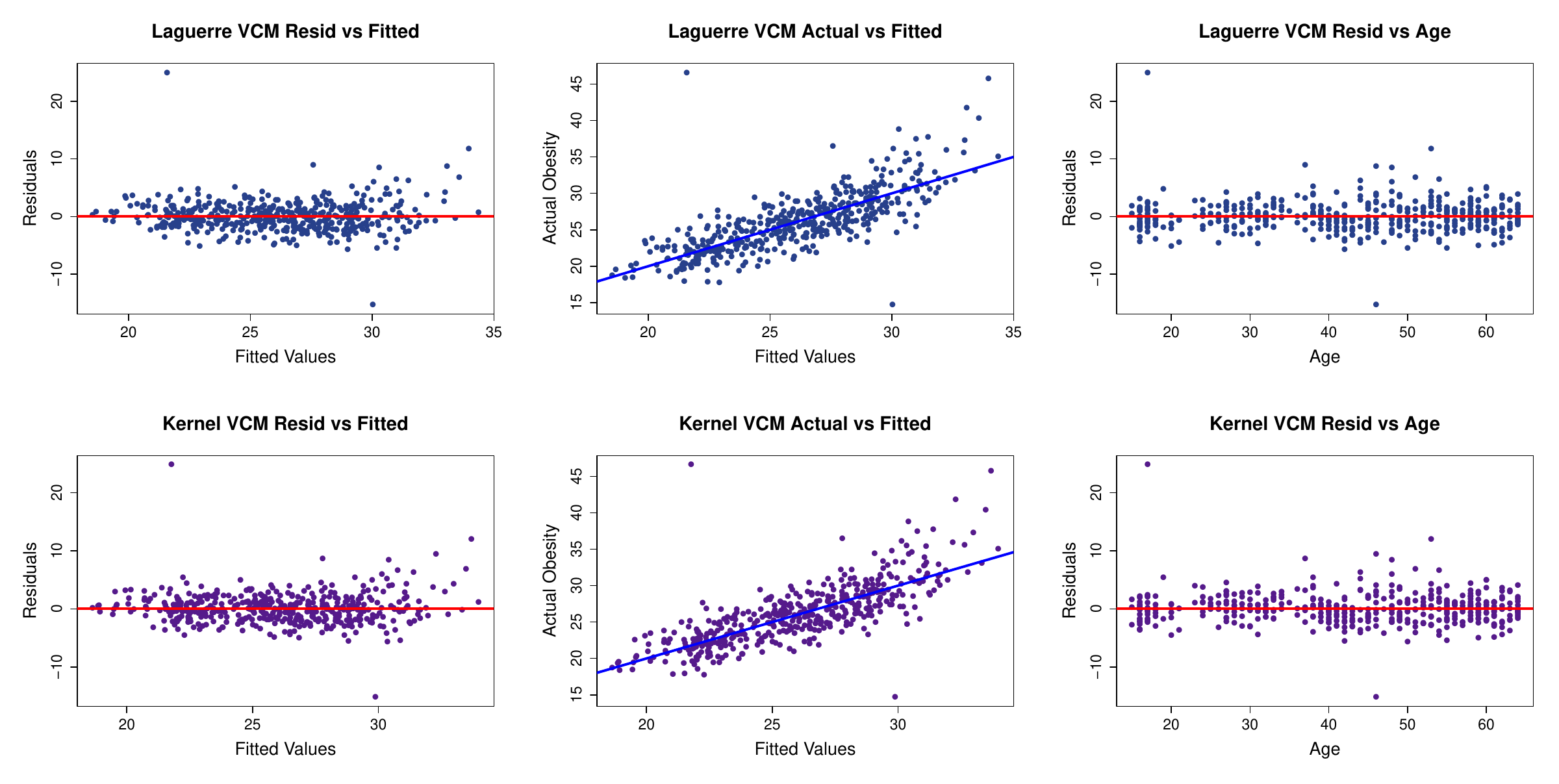}
    \caption{Laguerre and Local Kernel VCM Residual Analysis}
    \label{LaguerreandLKResidualanalysis}
\end{figure}
\begin{figure}
\caption{}
\includegraphics[width=6in]{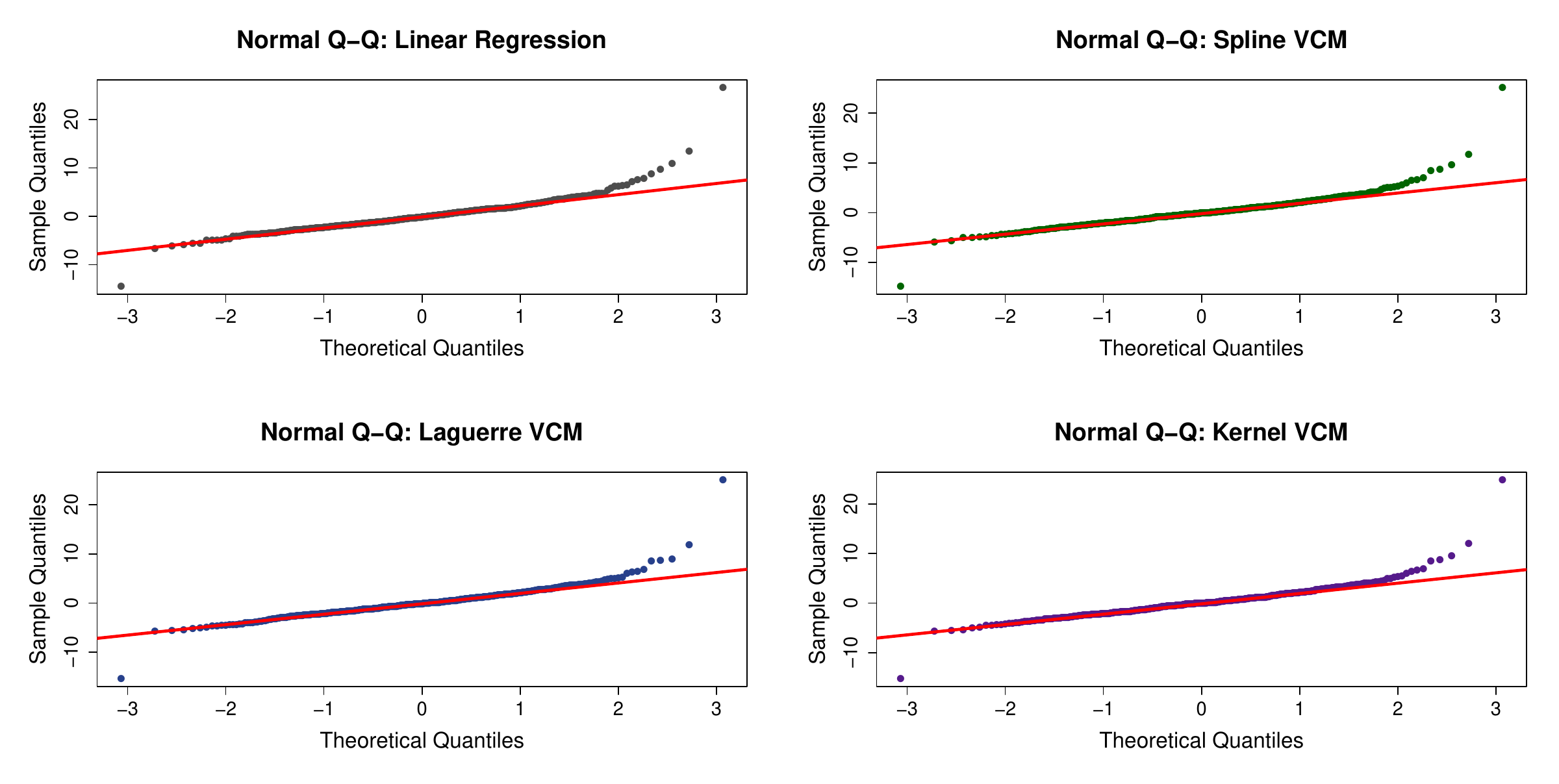}
  \label{qqplotsforall}
\end{figure}
Figure \ref{SSandLMResiduals} and \ref{LaguerreandLKResidualanalysis} visualize the errors and how they vary based on the variable, age. There is no evidence to suggest for the Generalized Laguerre, Local Linear Kernel, Smoothing Spline Methods or the linear regression, that errors aren't normally distributed. In addition, the Age versus Residuals plots for Local Linear estimator, the Laguerre estimator and the Smoothing Spline estimator show that VCM consideration is an adequate choice for this data set. Figure \ref{qqplotsforall} shows a Normal Probability plot that indicates it is approximately Gaussian. Only near the tails, there is heavy variance from the normal line. Therefore, the Normal Probability plots and behavior of the other residual plots indicate that the errors aren't dependent on age or dependent on $\hat{y}$ in any of the methods/models chosen to analyze.
\begin{figure}
\caption{}
\includegraphics[width=6in]{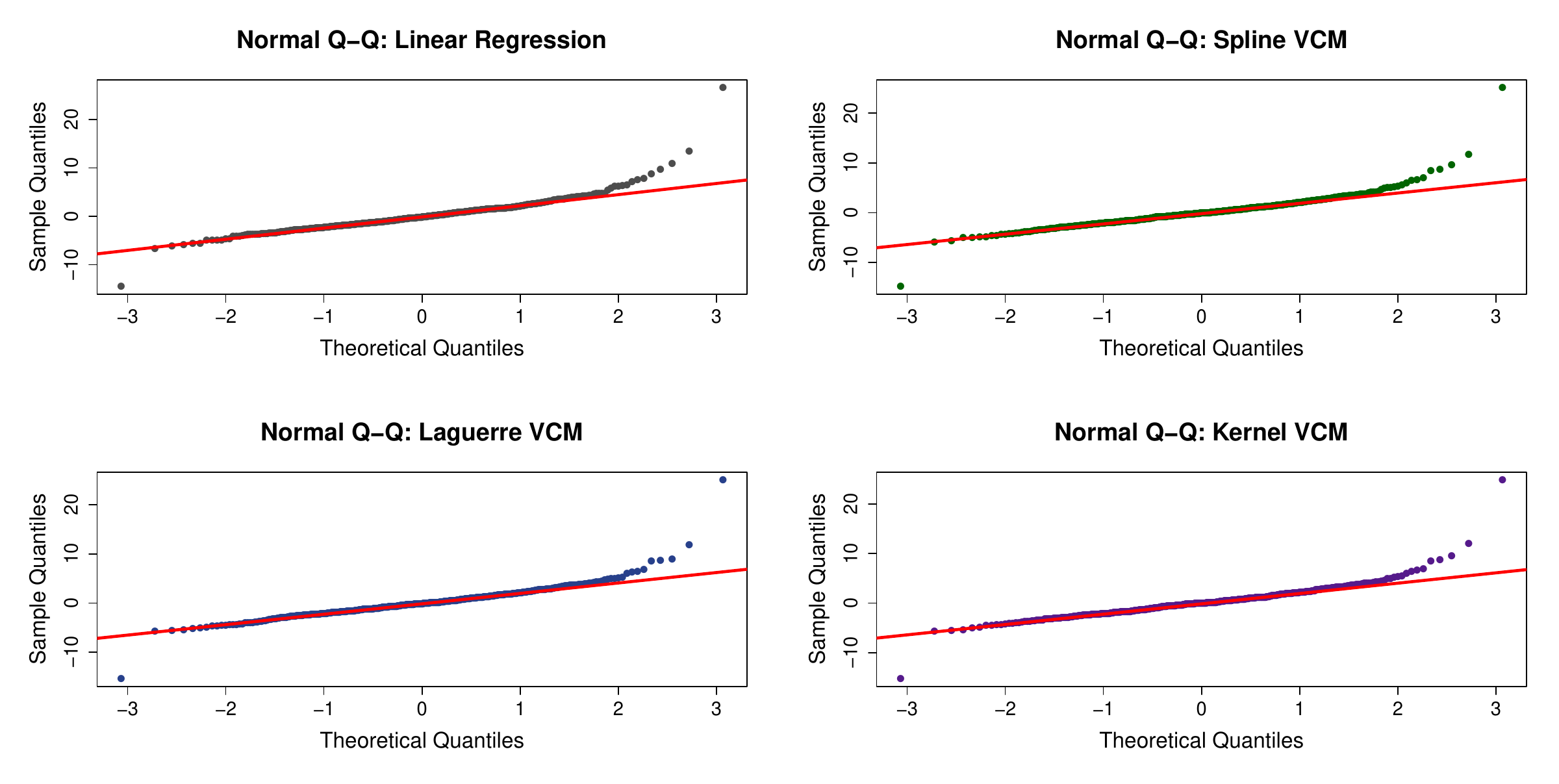}
  \label{qqplotsforall}
\end{figure}
\section{Proofs}
 {\bf Proof of Theorem \ref{th:upperbds-indiv}.} Notice that, by \fr{Y-syst-phi}, one has
\beqn \label{est-diff}
\widehat{\beta}_l(t)-\beta_l(t)=0-\rho_l(t)+\left(\tilde{\boldsymbol{\phi}}_l(t)\right)^T\left[\Phi^T\Phi\right]^{-1}\Phi^T\boldsymbol{\rho}+\sigma\left(\tilde{\boldsymbol{\phi}}_l(t)\right)^T\left[\Phi^T\Phi\right]^{-1}\Phi^T\boldsymbol{\varepsilon}_n.
\eeqn
Therefore, mean square integrated error for \fr{est-diff} is
\beqn\label{err-decomp}
\EE\|\widehat{\beta}_l(t)-\beta_l(t)\|^2_2=O\left( \sum^{\infty}_{k=M_l}\theta_{lk}^2+\EE\|\left(\tilde{\boldsymbol{\phi}}_l(t)\right)^T\left[\Phi^T\Phi\right]^{-1}\Phi^T\boldsymbol{\rho}\|^2_2+\sigma^2\EE\|\left(\tilde{\boldsymbol{\phi}}_l(t)\right)^T\left[\Phi^T\Phi\right]^{-1}\Phi^T\boldsymbol{\varepsilon}_n\|^2_2\right).
\eeqn
\\
For the third term in \fr{err-decomp}
\beqn \label{var-iid}
\EE\|\left(\tilde{\boldsymbol{\phi}}_l(t)\right)^T\left[\Phi^T\Phi\right]^{-1}\Phi^T\boldsymbol{\varepsilon}_n\|^2_2&=&\EE\left(\left(\tilde{\boldsymbol{\phi}}_l(t)\right)^T\left[\Phi^T\Phi\right]^{-1}\Phi^T\boldsymbol{\varepsilon}_n\boldsymbol{\varepsilon}_n^T\Phi\left[\Phi^T\Phi\right]^{-1}\left(\tilde{\boldsymbol{\phi}}_l(t)\right)\right)\nonumber\\
&=&\EE\left(\left(\tilde{\boldsymbol{\phi}}_l(t)\right)^T\left[\Phi^T\Phi\right]^{-1}\Phi^T\EE[\boldsymbol{\varepsilon}_n\boldsymbol{\varepsilon}_n^T]\Phi\left[\Phi^T\Phi\right]^{-1}\left(\tilde{\boldsymbol{\phi}}_l(t)\right)\right)_{h, \mu}\nonumber\\
&=&\EE\left(\left(\tilde{\boldsymbol{\phi}}_l(t)\right)^T\left[\Phi^T\Phi\right]^{-1}\Phi^T\Sigma_n\Phi\left[\Phi^T\Phi\right]^{-1}\left(\tilde{\boldsymbol{\phi}}_l(t)\right)\right)_{h, \mu}\nonumber\\
&\leq&\lambda_{\max}(\Sigma_n)\EE\left[\left(\tilde{\boldsymbol{\phi}}_l(t)\right)^T\left[\Phi^T\Phi\right]^{-1}\left(\tilde{\boldsymbol{\phi}}_l(t)\right)\right]_{h, \mu}.
\eeqn
 Similarly,
\beqn\label{rhorohbnd}
\EE\|\left(\tilde{\boldsymbol{\phi}}_l(t)\right)^T\left[\Phi^T\Phi\right]^{-1}\Phi^T\boldsymbol{\rho}\|^2_2&=&\EE\left[tr\left(\left(\tilde{\boldsymbol{\phi}}_l(t)\right)^T\left[\Phi^T\Phi\right]^{-1}\Phi^T\boldsymbol{\rho}\boldsymbol{\rho}^T\Phi\left[\Phi^T\Phi\right]^{-1}\left(\tilde{\boldsymbol{\phi}}_l(t)\right)\right)\right]_{h, \mu}\nonumber\\
&\leq&\EE\left[\lambda_{\max}\left(\boldsymbol{\rho}\boldsymbol{\rho}^T\right)tr\left(\left[\Phi^T\Phi\right]^{-1}{\bf I}_l\right)\right].
\eeqn
Furthermore, since the density $h\geq m_o$, and since for $\phi_k(t)$ defined in \fr{lague-def} is such that $\phi_k(t)\leq 1$, one has 
\beqn 
\rho^2_l(t_i)\leq \frac{1}{m_o}\left(\sum^{\infty}_{k=M_l}|\theta_{lk}|\right)^2
\leq \frac{1}{m_o}\sum^{\infty}_{k=M_l}k^2\theta^2_{lk}\sum^{\infty}_{k=M_l}\frac{1}{k^2}
\leq \frac{A_l}{m_oM_l}. 
\eeqn
This result means that \fr{rhorohbnd} can be bounded by 
\be
\EE\|\left(\tilde{\boldsymbol{\phi}}_l(t)\right)^T\left[\Phi^T\Phi\right]^{-1}\Phi^T\boldsymbol{\rho}\|^2_2\leq \frac{A}{m_oM_l}tr\left(\EE\left[\left[\Phi^T\Phi\right]^{-1}\right]_{h, \mu}{\bf I}_l\right).
\ee
Hence, by condition \fr{lagsob}, \fr{err-decomp} reduces to  \fr{err-singl}.\\
  {\bf Proof of Theorem \ref{th:upperbds-vectr}.} Left-multiplying \fr{Y-syst-phi} by $\left[\Psi(t)\right]^T$ and subtracting $\boldsymbol{\beta}(t)$, yields
\beqn \label{est-vect-diff}
\widehat{\boldsymbol{\beta}}(t)-\boldsymbol{\beta}(t)=-\rho(t)+\left[\Psi(t)\right]^T\left[\Phi^T\Phi\right]^{-1}\Phi^T\boldsymbol{\rho}+\sigma\left[\Psi(t)\right]^T\left[\Phi^T\Phi\right]^{-1}\Phi^T\boldsymbol{\varepsilon}_n.
\eeqn
Therefore, the expectation of the squared norm of \fr{est-vect-diff} gives
\beqn
\EE\|\widehat{\boldsymbol{\beta}}(t)-\boldsymbol{\beta}(t)\|^2&\asymp& \sum^{\bf r}_{l=1}\int^{\infty}_0\rho^2_l(t)h(t)dt+\EE\left[\lambda_{\max}\left[\boldsymbol{\rho}\boldsymbol{\rho}^T\right]tr\left(\left[\Psi(t)\right]^T\left[\Phi^T\Phi\right]^{-1}\left[\Psi(t)\right]\right)\right]\nonumber\\
&+& \sigma^2\lambda_{\max}\left(\Sigma_n\right)\EE\left[tr\left(\left[\Psi(t)\right]^T\left[\Phi^T\Phi\right]^{-1}\left[\Psi(t)\right]\right)\right].
\eeqn
Hence, similar to \fr{err-singl}, it can be shown that the second term is of order lower than the third. \\
{\bf Proof of Theorem \ref{th:lowerbds}}.
In order to prove the lower bounds, we use the following lemma.  
\begin{lemma} (Lemma $A.1$ of Bunea et al.~(2007)) 
\label{lem:Bunea} 
Let $\Te$ be a set of functions of cardinality $\card(\Te)\geq 2$ such that\\
(i) $\|f-g\|_2^2 \geq 4\delta^2, \ for\  f, g \in \Te, \ f \neq g, $\\
(ii) the Kullback divergences $\mathcal{K}(P_f, P_g)$ between the measures $P_f$ and $P_g$ 
satisfy the inequality $\mathcal{K}(P_f, P_g) \leq \log(\card(\Te))/16,\ for\ f,\ g \in \Te$.\\
Then, for some absolute positive constant $C_5$, one has 
$$
 \inf_{f_n}\sup_{f\in \Te} \EE_f\|f_n-f\|_2^2 \geq C_5 \delta^2,
$$
where $\inf_{f_n}$ denotes the infimum over all estimators.
\end{lemma}
Consider a set of  test vector functions $\boldsymbol{\beta}_{\omega}(t)=\left(\beta_{1, \omega}(t), \beta_{2, \omega}(t), \cdots, \beta_{{\bf r}, \omega}(t)\right)^T$ with components of the form 
\be\label{test-f}
\beta_{l, \omega}(t)=\lambda_{l}\sum^{M_l-1}_{k=0}\omega_{l, k}\tilde{\phi}_k(t),
\ee
where $\omega_{l, k}\in\left\{0, 1\right\}$, for $k=0, 1, 2, \cdots, M_l$, $l=1, 2, \cdots, {\bf r}$. Notice that in order for \fr{test-f} to satisfy \fr{lagsob} we set $\lambda_l=A_lM_l^{-(\gamma_l+1/2)}$. In addition, consider another binary sequence $\tilde{\omega}$ and its corresponding vector function  $\boldsymbol{\beta}_{\tilde{\omega}}(t)$. Keep in mind that the set $\boldsymbol{\Omega}$ has cardinality $Card(\boldsymbol{\Omega})=2^{\sum^{\bf r}_{l=1}M_l}$. Then, the total square distance in $L^2(0, \infty)$ is 
\beqn\label{dist-bet}
\|\boldsymbol{\beta}_{\omega}(t)-\boldsymbol{\beta}_{\tilde{\omega}}(t)\|^2_{h}=\sum^{\bf r}_{l=1}\gamma^2_l\sum^{M_l-1}_{k=0}|\omega_{l, k}-\tilde{\omega}_{l, k}|^2\geq \frac{\sum^{\bf r}_{l=1}A_l^2M_l^{-2\gamma_l}}{8}. 
\eeqn
Due to Varshamov-Gilbert Lemma (Tsybakov~(2008), page 104). Denote
\beqn\label{q-bet}
Q(\boldsymbol{\beta}_{\omega})(i)={\bf x}_i\boldsymbol{\beta}_{\omega}(t_i),\ \ i=1, 2, \cdots, n,
\eeqn
and let ${\bf Q}(\boldsymbol{\beta}_{\omega})$ be the vector with components \fr{q-bet}. Let ${\bf P}_{\boldsymbol{\omega}}$ be the probability law for the process \fr{conveq} under the hypothesis $\boldsymbol{\beta}_{\omega}$. Then, the Kullback divergence between ${\bf P}_{\boldsymbol{\omega}}$ and ${\bf P}_{\tilde{\boldsymbol{\omega}}}$ is
\beqns
\mathcal{K}\left({\bf P}_{\boldsymbol{\omega}}, {\bf P}_{\tilde{\boldsymbol{\omega}}}\right)&=& \frac{1}{2\sigma^2}\EE\left(\left[{\bf Q}(\boldsymbol{\beta}_{\omega})-{\bf Q}(\boldsymbol{\beta}_{\tilde{\omega}})\right]^T\Sigma_n^{-1}\left[{\bf Q}(\boldsymbol{\beta}_{\omega})-{\bf Q}(\boldsymbol{\beta}_{\tilde{\omega}})\right]\right)\\
&\leq& \frac{n}{2\sigma^2}\lambda_{\max}(\Sigma^{-1}_n)\EE\left([\boldsymbol{\beta}_{\omega}-\boldsymbol{\beta}_{\tilde{\omega}}]^T{\bf X}^T{\bf X}[\boldsymbol{\beta}_{\omega}-\boldsymbol{\beta}_{\tilde{\omega}}]\right)\\
&\leq& \frac{n^{\alpha}}{2c_1\sigma^2}\lambda_{\max}({\bf W})\|\boldsymbol{\beta}_{\omega}(t)-\boldsymbol{\beta}_{\tilde{\omega}}(t)\|^2_{h}\\
&\leq& \frac{n^{\alpha}}{2c_1\sigma^2}\lambda_{\max}({\bf W})\sum^{\bf r}_{l=1}A_l^2M_l^{-2\gamma_l}.
\eeqns
{\bf Lemma \ref{lem:Bunea} } implies that we set 
\beqn \label{lem-cond-ii}
\frac{n^{\alpha}}{2c_1\sigma^2}\lambda_{\max}({\bf W})\sum^{\bf r}_{l=1}A_l^2M_l^{-2\gamma_l}\leq \frac{\pi_o}{16} \sum^{\bf r}_{l=1}M_l.
\eeqn
Notice that based on \fr{lem-cond-ii}, one can choose $M_l$ as
\be \label{lem-bwd}
M_l\asymp \left[\frac{A_l^2n^{\alpha}}{c_1\sigma^2}\right]^{\frac{1}{2\gamma_l+1}}.
\ee
Hence,  plugging \fr{lem-bwd} in \fr{dist-bet} completes the proof. \\
{\bf Proof of Theorem \ref{thnor}}. The proof follows from the linearity of the estimator and the Gaussianity of the errors. 
$$\widehat{\beta}_l(t) - \beta_l(t) = - \rho_l(t) + \left( \tilde{\boldsymbol{\phi}}_l(t) \right)^T \left( \Phi^T \Phi \right)^{-1} \Phi^T \boldsymbol{\rho} + \left( \tilde{\boldsymbol{\phi}}_l(t) \right)^T \left( \Phi^T \Phi \right)^{-1} \Phi^T \boldsymbol{\varepsilon}_n.$$
Then we have 
\beqn\label{dechatbeta}
\sqrt{n^{\alpha}}\left(\widehat{\beta}_l(t) - \beta_l(t)  \right)&=& \sqrt{n^{\alpha}}\left(- \rho_l(t) + \left( \tilde{\boldsymbol{\phi}}_l(t) \right)^T \left( \Phi^T \Phi \right)^{-1} \Phi^T \boldsymbol{\rho} + \left( \tilde{\boldsymbol{\phi}}_l(t) \right)^T \left( \Phi^T \Phi \right)^{-1} \Phi^T\boldsymbol{\varepsilon}_n\right)\nonumber \\
&=:& \sqrt{n^{\alpha}}\left(- \rho_l(t) + W_l(t)\boldsymbol{\rho} + W_l(t) \boldsymbol{\varepsilon}_n\right)
\eeqn
Since $\boldsymbol{\varepsilon}_n = (\varepsilon_1, \varepsilon_2, \dots, \varepsilon_n)^T$ is a stationary Gaussian vector with mean zero and covariance matrix $\Sigma_n$,
where $\Sigma_{ij} = \mathrm{Cov}(\varepsilon_i, \varepsilon_j) $ for $\alpha \in (0, 1)$, $W_l(t)$ is also Gaussian and we have, $$\mathrm{E} \left[ W_l(t) \right] = \left( \tilde{\boldsymbol{\phi}}_l(t) \right)^T \left( \Phi^T \Phi \right)^{-1} \Phi^T \mathrm{E} [\boldsymbol{\varepsilon}_n] = 0.$$
The variance of $W_l(t)\boldsymbol{\varepsilon}_n$ is:
\beqn\mathrm{Var} \left( W_l(t) \boldsymbol{\varepsilon}_n\right) &=& \mathrm{E} \left[ (W_l(t)\boldsymbol{\varepsilon}_n) ^2 \right]\nonumber \\
& =& \mathrm{E} \left[ \left( \tilde{\boldsymbol{\phi}}_l(t) \right)^T \left( \Phi^T \Phi \right)^{-1} \Phi^T \boldsymbol{\varepsilon}_n \boldsymbol{\varepsilon}_n^T \Phi \left( \Phi^T \Phi \right)^{-1} \left( \tilde{\boldsymbol{\phi}}_l(t) \right) \right]. \nonumber
\eeqn
Using the linearity of expectation and the fact that $\EE[\boldsymbol{\varepsilon}_n \boldsymbol{\varepsilon}_n^T]=\Sigma_n$:
\be\label{varWeps}\mathrm{Var} \left( W_l(t) \boldsymbol{\varepsilon}_n\right) = \left( \tilde{\boldsymbol{\phi}}_l(t) \right)^T \left( \Phi^T \Phi \right)^{-1} \Phi^T \Sigma_n \Phi \left( \Phi^T \Phi \right)^{-1} \left( \tilde{\boldsymbol{\phi}}_l(t) \right).
\ee
Recall that the $\{\varepsilon_i\}$ suffers from long-memory and $
\Var\left(\sum^N_{i=1}\varepsilon_i\right)\sim \pi_{\alpha}n^{2-\alpha},
$
where $\pi_{\alpha}$ is a finite and positive constant. By using theorem 2.1 of Taqqu~(1975) we obtain 
\beqn\label{Taqlim}
\lim_{n \to \infty} \frac{1}{n^{2-\alpha}} \Phi^T \Sigma_n \Phi \xrightarrow{p} \pi_\alpha \Gamma.\eeqn
In addition, for all $0\le j \le M_l$, and $0\le k \le M_l$, notice that
\beqn\left(\frac{\Phi^T \Phi}{n}\right)_{jk} &=&  \frac{1}{n}\sum_{i=1}^n \Phi_{ik} \Phi_{ij}\nonumber \\
&=&\frac{1}{n} \sum_{i=1}^n X_{li} X_{lj} \tilde{\phi}_k(t_i) \tilde{\phi}_j(t_i).\nonumber
\eeqn
Due to the law of large numbers, guaranteed by the fact that the design points $(t_i, x_i)$ are $i.i.d.$ with densities $h(t)$ and $\mu(x)$, one has
\beqn\label{Gammajk}
\lim_{n \to \infty}\left(\frac{\Phi^T \Phi}{n}\right)_{jk} &=& \mathbb{E} \left[ X_{l1} X_{l1} \tilde{\phi}_k(t_1) \tilde{\phi}_j(t_1) \right]=:\Gamma_{jk}
\eeqn
Now, combining (\ref{varWeps}), (\ref{Taqlim}) and (\ref{Gammajk}),  yields 
\be\label{varlim} 
n^{\alpha} \cdot \mathrm{Var} \left( W_l(t) \boldsymbol{\varepsilon}_n\mid t\right)  \xrightarrow{p} \pi_{\alpha} \tilde{\boldsymbol{\phi}}_l(t) ^T \Gamma^{-1}   \tilde{\boldsymbol{\phi}}_l(t).
\ee
Therefore, conditional on $t$, one has
\be\label{limW_eps}
n^{\alpha/2} \cdot W_l(t) \boldsymbol{\varepsilon}_n \xrightarrow{d} \mathcal{N}\left(0,\pi_{\alpha} \tilde{\boldsymbol{\phi}}_l(t) ^T \Gamma^{-1}   \tilde{\boldsymbol{\phi}}_l(t)\right)
\ee
The remainder terms (involving $\rho_l$) are shown to be $o_p(1)$ under the Laguerre-Sobolev assumption $\|\rho_l\|_h^2 \leq A^2 M_l^{-2\gamma_l}$, with $M_l \asymp n^{1/(2\gamma_l + 1)}$. We have 
\beqn\label{rho}
n^{\alpha/2} \cdot\|\rho_l\|_h &\le& An^{\alpha/2}M_l^{-\gamma_l}
\eeqn
leading to
\be\label{srho}
n^{\alpha/2} \rho_l = O(n^{\alpha/2-\gamma_l/(2\gamma_l + 1)})
\ee
which vanishes as $n \to \infty$ since $\alpha/2<\gamma_l/(2\gamma_l + 1)$. Using the bounds derived in equation (22), and since $\lambda_{\max}(\boldsymbol{\rho}\boldsymbol{\rho}^T) \leq \|\boldsymbol{\rho}\|_h^2 \leq {\bf r}^2\sum^{\bf r}_{l=1}A_l^2 M_l^{-2\gamma_l}$, with $\gamma_o=\min\left\{\gamma_1, \gamma_2, \cdots, \gamma_{\bf r}\right\}$, it follows that, 
 \be\label{W_lrho}\sqrt{n^{\alpha}} \cdot \left( \tilde{\boldsymbol{\phi}}_l(t) \right)^T \left( \Phi^T \Phi \right)^{-1} \Phi^T \boldsymbol{\rho} = O \left( n^{\alpha/2 - \gamma_o / (2\gamma_o + 1)} \right) = o_p(1).
 \ee
 Finally, combining (\ref{dechatbeta}), (\ref{limW_eps}), (\ref{srho}), (\ref{W_lrho}) and  Slutsky’s theorem completes the proof.
 

\begin{thebibliography}{99}
\bibitem{ben-Conn} Benhaddou, R., Connell and M. L. (2022),'Nonparametric empirical Bayes estimation based on generalized Laguerre series', {\it Communications in Statistics-Theory and Methods}, {\bf 52(19)}, 6896-6915.

\bibitem{ben-pensk-raj} Benhaddou, R., Pensky, M., and Rajapaksha, R. (2017),'Anisotropic functional Laplace deconvolution', {\it Journal of Statistical Planning and Inference}, {\bf 199}, 271-285.

\bibitem{bun-tsy-Weg} Bunea, F., A. Tsybakov, and M. H. Wegkamp. (2007),'Aggregation for Gaussian regression', {\it The Annals of Statistics}, {\bf 35(4)}, 1674-1697.

\bibitem{cai-fan-li} Cai, Z., Fan, J., and Li, R. (1999),'Efficient estimation and inferences for varying-coefficient models', {\it Journal of the American Statistical Association}, {\bf 95}, 888-902.

\bibitem{Chia-Ri-Wu} Chiang, C. T., Rice, J. A., and Wu, C. O. (2001),'Smoothing spline estimation for varying coefficient models with repeatedly measured dependent variables', {\it Journal of the American Statistical Association}, {\bf 96}, 605-619.

\bibitem{comt-gen} Comte, F., Genon-Catalot, V. (2015), 'Adaptive Laguerre density estimation for mixed Poisson models', {\it The Electronic Journal of  Statistics}, {\bf 9(1)}, 1113-1149.

\bibitem{comt-c-p-r} Comte, F., Cuenod, C. A., Pensky, M., Rozenholc Y. (2017), 'Laplace deconvolution on the basis of Time Domain Data and its Application to Dynamic Contrast-Enhanced Imaging', {\it The Journal of the Royal Statistical Society Series B}, {\bf 79(1)}, 69-94.

\bibitem{dusp} Dussap, F. (2021),'Anisotropic multivariate deconvolution using projection on the Laguerre basis', {\it Journal of Statistical Planning and Inference}, {\bf 215(3)}, 23-46.

\bibitem{fan-zhan} Fan, J., Yao Q., and Cai, Z. (2003),'Adaptive varying-coefficient linear models', {\it The Journal of the Royal Statistical Society B}, {\bf 65}, 57-80.

\bibitem{fan-zhan1} Fan, J., and Zhang, W. (1999),'Statistical estimation in varying coefficient models', {\it The Annals of Statistics}, {\bf 27(5)}, 1491-1518.

\bibitem{fan-zhan2} Fan, J., and Zhang, W. (2008),'Statistical Methods with varying coefficient models', {\it Statistical Interface}, {\bf 1(1)}, 179-195.

\bibitem{grtn-rk} Gradshtein, I. S., and Ryzhik, I. M. (1980), 'Tables of integrals, series, and products', {\it New York: Academic Press}.

\bibitem{Hastie-Tibshi} Hastie, T. J., and Tibshirani, R. J. (1993),'Varying-coefficient models', {\it Journal of the Royal Statistical Society: B}, {\bf 55}, 757-796.

\bibitem{hoov-Ri-W-y} Hoover, D.R., Rice, J. A., Wu, C.O., and Yang, L. P. (1998),'Nonparametric smoothing estimates of time-varying coefficient models with longitudinal data', {\it Biometrika}, {\bf 85}, 809-822.

\bibitem{huan-wu-zho} Huang, J.Z., Wu, C.O., and Zhou, L.(2004),'Polynomial spline estimation and inference for varying coefficient models with longitudinal data', {\it Statistica Sinica}, {\bf 14}, 763-788.

\bibitem{lepsmamsp} Laverny, O., Masiello, E.,  Maume-Deschamps, V., and Rulliere, D.(2021), 'Estimation of multivariate generalized gamma deconvolution through Laguerre expansions', {\it The Electronic Journal of  Statistics}, {\bf 15(2)}, 5158-5202.

\bibitem{pensk-klopp1} Park, B. U., Mammen, E., Lee, Y. K., and Lee, E. R. (2015),'Varying coefficient model: A Review and New Development', {\it International Statistical Review}, {\bf 83(1)}, 36-64.

\bibitem{pensk-klopp1} Pensky, M., and Klopp, O. (2013),'Non-asymptotic approach to varying coefficient Regression Models', {\it The Electronic Journal of Statistics}, {\bf 7}, 454-479.

\bibitem{pensk-klopp2} Pensky, M., and Klopp, O. (2015),'Sparse high-dimensional varying coefficient mode: Non-asymptotic minimax study', {\it The Annals of Statistics}, {\bf 43(3)}, 1273-1299.

\bibitem{Taq} Taqqu, M. (1975), 'Weak convergence to fractional Brownian motion and to Rosenblatt process', {\it Zeitschrift fur Wahrscheinlichkeitstheorie und Verwandte Gebiete}, {\bf 31}, 287-302.

\bibitem{Tsybv} Tsybakov, A. B. (2008), 'Introduction to nonparametric estimation', {\it New York: Springer}.

\bibitem{varesc} Vareschi, T. (2015), 'Noisy Laplace deconvolution with error in the operator', {\it Journal of Statistical Planning and Inference}, {\bf 157-158}, 16-35.

\bibitem{Wu-Chia} Wu, C.O., Chiang, C. T., and Hoover, C. T. (1998),'Asymptotic confidence regions for kernel smoothing of a varying-coefficient model with longitudinal data', {\it Journal of the American Statistical Association}, {\bf 93}, 1388-1402.

\bibitem{Zhou-You} Zhou, X., and You, J. (2004),'Wavelet estimation in varying-coefficient partially linear regression models', {\it Statistics and Probability Letters}, {\bf 68}, 91-104.
\end{thebibliography}
\end{document}